\documentclass[10pt]{amsart}

\usepackage{amsfonts}
\usepackage{amssymb}
\usepackage{amsmath}
\usepackage{amsaddr}

\newtheorem{theorem}{Theorem}[section]
\newtheorem{proposition}[theorem]{Proposition}
\newtheorem{corollary}[theorem]{Corollary}
\newtheorem{scholium}[theorem]{Scholium}

\newtheorem{conjecture}[theorem]{Conjecture}
\newtheorem{example}[theorem]{Example}
\newtheorem{remark}[theorem]{Remark}
\newtheorem{definition}[theorem]{Def\mbox{}inition}

\title[Self-extension property]{Non-Hilbert Banach spaces with the self-extension property}

\author[J.B.Garc\'ia]{Juan Bosco Garc\'ia-Guti\'errez}
\address{Department of Mathematics, University of C\'adiz, Puerto Real, Spain, 11510, EU}
\email{juan.bosco@uca.es}

\author[F.J.Garc\'ia]{Francisco Javier Garc\'{i}a-Pacheco}
\address{Department of Mathematics, University of C\'adiz, Puerto Real, Spain, 11510, EU}
\email{garcia.pacheco@uca.es}

\author[P.Piniella]{Paula Piniella}
\address{Department of Mathematics, University of C\'adiz, Puerto Real, Spain, 11510, EU}
\email{piniella5050@gmail.com}

\author[F.Rambla]{Fernando Rambla-Barreno}
\address{Department of Mathematics, University of C\'adiz, Puerto Real, Spain, 11510, EU}
\email{fernando.rambla@uca.es (corresponding author)}

\begin{document}

\begin{abstract}
In 1992, Kiendi, Adamy and Stelzner investigated under which conditions a certain type of function constituted a Lyapunov function for some time-invariant linear system. Six years later, it was obtained that this property holds if and only if the Banach space enjoys the self-extension property. However, the knowledge of these spaces needed to be extended in order to make useful this characterization, since there were little information on which classic Banach spaces satisfy this property and its relations with other classic properties. We present the self-extension property in a wider frame, relating it to other well-known spaces as the $1$-injective or $1$-projective ones. We investigate the property in two important low-dimensional classic Banach spaces: $\mathbb{R}_1^3$ and $\mathbb{R}_1^4$. We introduce the concept of $k-$self-extensible spaces and a discussion of the stability of the property. We also show that every real Banach space of dimension greater than or equal to $3$ can be equivalently renormed to fail the self-extension property. Finally, we summarize some consequences of our study and mention some open questions which appear naturally.
\end{abstract}


\maketitle

\vspace{3ex}

\section{Introduction}

In \cite{KAS1992}, the authors look for conditions to make a function def\mbox{}ined as $$V:\mathbb{R}^n \to \mathbb{R} $$ $$V(x)=\|W x\|$$ (where $W\in \mathcal{M}_{m\times n}(\mathbb{R})$ has rank $n$) be a Lyapunov function for some time-invariant linear system, ending up with a characterization of this property. However, in \cite{H1994} it is highlighted that the result of the previous paper does not constitute a characterization but only gives a suf\mbox{}f\mbox{}icient condition. Moreover, in \cite{LPR1998} is proven that the necessity only holds in a certain class of normed spaces (those satisfying a property called self-extension which will be def\mbox{}ined below).

To be precise, the following was proved:
\begin{theorem}[\cite{LPR1998}]
Let $\|\ \|$ be a norm in $\mathbb{R}^m$. The following assertions are equivalent:
\begin{itemize}
\item $(\mathbb{R}^m,\|\ \|)$ has the self-extension property.
\item Let $V: \mathbb{R}^n \to \mathbb{R}$ be given by $V(x) = \|Wx\|$ where $W \in {\mathbb{R}}^{m \times n}$ is of full rank. Then $V$ 
is a Lyapunov function for a discrete-time system $x_{t+1}=F x_t$, where $F \in {\mathbb{R}}^{n \times n}$, if and only if there exists a matrix $Q \in {\mathbb{R}}^{m \times m}$ such that $WF=QW$ and $\|Q\| < 1$.
\end{itemize}
\end{theorem}

They also prove that Hilbert spaces and $\ell_\infty(\Gamma)$ spaces satisfy this property and construct a 3-dimensional example that fails it. However, there is neither a deep study of which classical spaces are self-extensible nor a research on the relation between the self-extension property and other classical properties. The aim of this paper is to go further on these directions.

In Section \ref{defs} we present the def\mbox{}inition of the self-extension property and we introduce the concepts of $1$-injective and $1$-projective space, linking them to self-extensible spaces. Since they are closely related to self extensible spaces, Sections 3 and 4 discuss and explore the relations between $1$-complemented spaces and norm-attaining functionals and ref\mbox{}lexive spaces. The main results are found in Section 5 and 6, where it is proven that $(\mathbb{R}^3,\|\cdot\|_1)$ satisf\mbox{}ies the self-extension property but $(\mathbb{R}^4,\|\cdot\|_1)$ does not; a discussion on the stability of the property is also included.

We will deal mainly with real Banach spaces, which is the setting where the property was originally def\mbox{}ined, although many results can be easily translated to the complex case. In the sequel, we use $\ell_p^n$ to denote either the real or the complex $n$-dimensional space with the norm $\|\cdot\|_p$, whereas $\mathbb{R}_p^n$ is to be used for the real case only.

\section{Basic def\mbox{}initions and results}\label{defs}

\begin{definition} A Banach space $X$ is said to have the {\bf self-extension property} if for every $Y$ subspace of $X$ and every $T:Y\to Y$ continuous linear operator there exists a continuous linear extension $S:X\to X$ such that $S|_Y=T$ and $\|S\|=\|T\|$.
\end{definition}

\begin{definition} Let $X$ be a Banach space and $k\geq 1$. A subspace $Y$ is said to be {\bf [$\pmb{k}$-]complemented} in $X$ if there exists a projection $P:X\to Y$ [such that $\|P\|\leq k$. We also say $P$ is a $k$-projection]. 
\end{definition}

We remind the reader that a projection on a Banach space $X$ is a continuous, linear and idempotent map $P:X\to X$. The dual operator $P^*:X^*\to X^*$ is also a projection. The complementary projection of $P$ is def\mbox{}ined as $I-P$, which is also a projection. Every nonzero projection has norm greater than or equal to $1$.

\begin{definition} Let $k \geq 1$. A Banach space $X$ is called a {\bf [$\pmb{k}$-]projective Banach space} if every subspace $Y$ of $X$ is [$k$-]complemented in $X$.
\end{definition}

\begin{definition} Let $k \geq 1$. A Banach space $X$ is called a {\bf [$\pmb{k}$-]injective Banach space} if it is [$k$]-complemented in every $Y$ such that $X \subseteq Y$.
\end{definition}

Notice that they are analogous def\mbox{}initions, switching subspaces with superspaces.

\begin{remark}\label{charinj}
In accordance with \cite[Page 123]{D1973}, a Banach space $X$ is $k$-injective if and only if satisf\mbox{}ies any of the following (for all arbitrary Banach spaces $Y$, $Z$):
\begin{enumerate}
\item If $X\subseteq Y$ and $T:X\to Z$ is linear and continuous, then there exists a continuous linear extension $S:Y\to Z$ with $\|S\|\leq k\|T\|$.
\item If $Z\subseteq Y$ and $T:Z\to X$ is linear and continuous, then there exists a continuous linear extension $S:Y\to X$ with $\|S\|\leq k\|T\|$.
\item If $T: X \to Y$ is a linear isometry, there exists a continuous linear $S: Y \to X$ such that $\|S\| \leq k$ and $S \circ T$ is the identity.
\end{enumerate}
If $k=1$, then  $\|S\|\leq k\|T\|$ can be replaced by  $\|S\|=\|T\|$ in {\em (1,2)} and $\|S\| \leq k$ can be replaced by $\|S\| =1$ in {\em (3)}.
\end{remark}

Moreover, some of the previous classes are completely characterized:
\begin{itemize}
\item The $1$-injective spaces are exactly the $\mathcal{C}(K)$ spaces for $K$ extremally disconnected, i.e., the closure of every open set in $K$ is open (Kelley \cite{Ke1952}).
\item A Banach space is $1$-projective if and only if it is a Hilbert space or 2-dimensional (Kakutani \cite{Ka1939} and Bohnenblust \cite{Bo1942}).
\item The projective spaces are exactly the spaces linearly isomorphic to a Hilbert space (Lindenstrauss and Tzafriri \cite{LT}).
\end{itemize}

All the previous characterizations are very deep results in Banach space theory. However, in \cite{Ph1940} it was observed that the Hahn-Banach Theorem applied coordinatewise shows that $\ell_\infty(\Gamma)$-spaces are $1$-injective (this is a particular and simple case of the aforementioned Kelley's result).

\begin{theorem}\label{thsuffauto}
Let $X$ be a Banach space. If $X$ is $1$-injective or $1$-projective, then $X$ satisf\mbox{}ies the self-extension property.
\end{theorem}

\begin{proof}
We will distinguish between the cases of $1$-injective and $1$-projective:

\begin{itemize}
\item Suppose f\mbox{}irst that $X$ is $1$-injective. Let $Y$ be a closed subspace of $X$ and $T:Y\to Y$ a continuous linear operator. Since the inclusion $i: Y\hookrightarrow X$ is a monomorphism, by Remark \ref{charinj}(2), there exists a continuous linear operator $S:X\to X$ such that $S\circ i= i\circ T$ and $\|S\|\leq \|i\circ T\|$. Note that $S|_Y=T$, so $S$ is an extension of $T$, hence $\|S\|\geq \|T\|$. On the other hand, $\|S\|\leq \|i\circ T\|\leq \|i\|\|T\|=\|T\|$. As a consequence, $\|S\|=\|T\|$.

\item Assume now that $X$ is $1$-projective. Let $Y$ be a closed subspace of $X$ and $T:Y\to Y$ a continuous linear operator. By hypothesis, $Y$ is the range of a norm-one projection $P:X\to X$. Now $S = T\circ P:X\to X$ satisf\mbox{}ies $S|_Y=T$ and $$\|T\|=\left\|(T\circ P)|_Y\right\|\leq \|S\|\leq \|T\|\|P\| \leq \|T\|.$$
\end{itemize}
\end{proof}

Note that the comments $1)$ and $2)$ in section IV of \cite{LPR1998} can be seen as particular cases of the previous theorem.

\begin{example}
Some instances of Banach spaces satisfying the self-extension property were pointed out at the introduction. We will provide two examples separating the properties $1$-injective and $1$-projective.
\begin{itemize}
\item $\ell_\infty$ is a $1$-injective Banach space which is not even projective, since $c_0$ is not complemented in $\ell_\infty$ (\cite{Ph1940}).
\item Any Hilbert space of dimension strictly greater than $1$ is $1$-projective but not $1$-injective.
\end{itemize}
\end{example}

\begin{proposition}
Let $X$ be a $1$-projective Banach space and let $Y$ be a $1$-injective closed subspace of $X$. Then either $\dim(Y)=1$ or $X=Y=\ell_\infty^2$.
\end{proposition}

\begin{proof}
We will use the aformentioned characterizations of $1$-injective and $1$-projective spaces.

Assume $\dim(Y)>1$. Since $X$ is isomorphic to a Hilbert space, we deduce that $Y$ must be a ref\mbox{}lexive $\mathcal{C}(K)$ space and therefore isometrically isomorphic to $\ell_\infty^n$ for some $n>1$. This in turns implies that $X$ is not a Hilbert space, so it is necessarily $2$-dimensional and thus $Y=X=\ell_\infty^2$.
\end{proof}

A direct consequence of the previous proposition is the following corollary.

\begin{corollary} If $X$ is a $1$-projective, $1$-injective Banach space, then either $X=\ell_\infty^1$ or $X=\ell_\infty^2$.
\end{corollary}

Notice that, in opposition to the previous proposition, there can be found $1$-injective Banach spaces containing $1$-projective Banach spaces.

\begin{example}
$\ell_\infty$ is a $1$-injective Banach space that contains an isometric copy of $\ell_2$, which is Hilbert and then $1$-projective.
\end{example}

\section{Suf\mbox{}f\mbox{}icient condition for selfextensionality}

Recall that the annihilator of a non-empty subset $M$ of a vector space is the set $M^\perp:=\{x^*\in X^*:M\subseteq \ker(x^*)\}$. The pre-annihilator of a non-empty subset $N$ of $X^*$ is def\mbox{}ined as $N^\top:=N^\perp\cap X=\bigcap_{n^*\in N}\ker(n^*)$.

We also remind the reader that $\mathsf{NA}(X)$ stands for the set of norm-attaining functionals on a normed space $X$.

\begin{remark}\label{key}
If $X$ and $Y$ are normed spaces and $T:X\to Y$ is linear and continuous, then $\overline{T(\mathsf{B}_X)}=\mathsf{B}_Y$ if and only if $T^*:Y^*\to X^*$ is an isometry. Now, let $X$ be a Banach space and $Y$ a $1$-complemented subspace of $X$. If $P:X\to Y$ is a $1$-projection, then $P^*(\mathsf{NA}(Y))\subseteq \mathsf{NA}(X)$ and $P(\mathsf{B}_X)=\mathsf{B}_Y$, therefore $P^*:Y^*\to X^*$ is an isometry, and hence if $\mathsf{NA}(Y)$ contains an inf\mbox{}inite dimensional vector subspace, then so does $\mathsf{NA}(X)$. Also, if $i:Y\hookrightarrow X$ denotes the inclusion, then $\ker(P)=\left((P^*\circ i^*)(X^*)\right)^\top$.
\end{remark}

We will make use of Remark \ref{key} to f\mbox{}ind a suf\mbox{}f\mbox{}icient condition for selfextensionality. Nevertheless, the previous remark applied to $1$-injective spaces has strong consequences relative to the linear structure of the norm-attaining functionals.

\begin{theorem}
Let $X$ be a Banach space. If $X$ contains an isometric copy $Y$ of an inf\mbox{}inite dimensional $1$-injective space, then $\mathsf{NA}(X)$ contains an inf\mbox{}inite dimensional vector subspace.
\end{theorem}

\begin{proof}
Since $Y$ is $1$-injective, $Y$ is $1$-complemented in $X$ so let $P:X\to Y$ be a $1$-projection. By Remark \ref{key} we have that $P^*(\mathsf{NA}(Y))\subseteq \mathsf{NA}(X)$. Since $Y$ is linearly isometric to a $\mathcal{C}(K)$ with $K$ compact, Hausdorf\mbox{}f, extremally disconnected and inf\mbox{}inite, we deduce that $\mathsf{NA}(Y)$ contains an inf\mbox{}inite dimensional vector subspace by virtue of \cite[Theorem 2.1]{AAAGP}. As a consequence, the condition $P^*(\mathsf{NA}(Y))\subseteq \mathsf{NA}(X)$ implies that $\mathsf{NA}(X)$ also contains an inf\mbox{}inite dimensional vector subspace in view of Remark \ref{key}.
\end{proof}

Since every equivalent norm on a subspace can be extended to an equivalent norm on the whole space, the previous theorem has an immediate corollary.

\begin{corollary}
Let $X$ be a Banach space. If $X$ contains an isomorphic copy $Y$ of an inf\mbox{}inite dimensional $1$-injective space, then $X$ can be equivalently renormed to make $\mathsf{NA}(X)$ contain an inf\mbox{}inite dimensional vector subspace.
\end{corollary}

Remark \ref{key} also motivates the following def\mbox{}inition. Notice that, if $Z\subseteq X^*$, then $Z^\top= \{0\}$ means that $X^*=\{0\}^\perp=\left(Z^\top\right)^\perp=\mathrm{cl}_{w^*}(Z)$. Conversely, if $\mathrm{cl}_{w^*}(Z)=X^*$, then $\left(Z^\top\right)^\perp=\mathrm{cl}_{w^*}(Z)=X^*$, so the Hahn-Banach Theorem assures that $Z^\top=\{0\}$.

\begin{definition}
Let $X$ be a Banach space and $Y$ a closed subspace of $X$. A $w^*$-closed subspace $Z$ of $X^*$ is said to be norm-attaining for $Y$ provided that for every $y\in \mathsf{S}_Y$ there is an element $z^*\in \mathsf{S}_Z$ such that $z^*(y)=1$. If $Z^\top\neq \{0\}$, then we will say that $Z$ is nontrivial.
\end{definition}

\begin{example}\label{nme}
Let $X$ be a Banach space. Let $y\in X$ with $\|y\|=1$ and $y^*\in\mathsf{S}_{X^*}$ so that $y^*(y)=1$. Note that     $Y:=\mathbb{K} y$ is a $1$-dimensional subspace of  $X$ and  $Z:=\mathbb{K} y^*$ is norm-attaining for $Y$ with $Z^{\top}=\ker(y^*)\neq \{0\}$. Notice also that the natural projection $P:\mathbb{K} y\oplus \ker(y^*)\to \mathbb{K} y$ has norm $1$.
\end{example}

The following theorem characterizes the existence of nontrivial norm-attaining subspaces. It is based upon the ideas provided by Remark \ref{key} and Example \ref{nme}.

\begin{theorem}\label{char}
Let $X$ be a Banach space. Let $Y$ be a closed subspace of $X$. If there exists  $Z \subseteq X^*$ norm-attaining for $Y$, then:
\begin{enumerate}
\item $Y\cap Z^{\top}=\{0\}$.
\item $\|y\|\leq \|y+z\|$ for all $y\in Y$ and all $z\in Z^\top$.
\item $Y+Z^\top$ is closed in $X$.
\item If, in addition,  $Y^\perp\cap Z=\{0\}$, then $Y+Z^\top=X$.
\item If, in addition, $Y$ is smooth, then $\mathsf{NA}(Y)\subseteq i^*(Z)$ and thus $(i^*)^{-1}(i^*(Z))$ is dense in $X^*$, where $i: Y\hookrightarrow X$ is the canonical inclusion.
\end{enumerate}
Conversely, if there exists a closed subspace $V\neq \{0\}$ of $X$ such that $Y\cap V =\{0\}$ and the natural projection $Y+ V \to Y$ has norm $1$, then there exists a nontrivial norm-attaining $Z \subseteq X^*$ for $Y$ with $V\subseteq Z^\top$ and $V=Z^\top\cap (Y+V)$. If, in addition, $Y+V=X$, then $Y^\perp \cap Z=\{0\}$.
\end{theorem}

\begin{proof}
Recall that $Z$ is $w^*$-closed and every $y\in \mathsf{S}_Y$ has an element $z^*\in \mathsf{S}_Z$ such that $z^*(y)=1$.
\begin{enumerate}

\item Take any $y\in Y\cap Z^\top$. Since $y\in Y$, there exists $z^*\in\mathsf{S}_Z$ with $z^*(y)=\|y\|$. Since $y\in Z^\top$, $z^*(y)=0$. As a consequence, $y=0$.

\item Fix any arbitrary $y\in Y,z\in Z^\top$. Since $Z$ is norm-attaining for $Y$, there exists $z^*\in\mathsf{S}_Z$ with $z^*(y)=\|y\|$. Note that $z^*(z)=0$. Then $$\|y\|=z^*(y)=z^*(y+z)\leq \|y+z\|.$$

\item Let $(y_n+z_n)_{n\in\mathbb{N}}$ be a sequence in $Y+Z^\top$ converging to some $x\in X$. By (2), $$\|y_p-y_q\|\leq \|(y_p-y_q)+(z_p-z_q)\|=\|(y_p+z_p)-(y_q+z_q)\|$$ for all $p,q\in\mathbb{N}$, therefore $(y_n)_{n\in\mathbb{N}}$ is a Cauchy sequence in $Y$, thus convergent  to some $y\in Y$. Since $(y_n+z_n)_{n\in\mathbb{N}}$ is convergent to $x$, we deduce that $(z_n)_{n\in\mathbb{N}}$ is convergent to $x-y$, so $x-y\in Z^\top$. Finally, $x=y+(x-y)\in Y+Z^\top$.

\item Observe that $(Y+Z^\top)^\perp =Y^\perp\cap \left(Z^\top\right)^\perp = Y^\perp\cap \overline{Z}^{w^*}=Y^\perp \cap Z=\{0\}$, which implies that $Y+Z^\top=\mathrm{cl}\left(Y+Z^\top\right)=\left((Y+Z^\top)^\perp\right)^\top =\{0\}^\top =X$. 

\item Let $y^*\in \mathsf{NA}(Y)\setminus\{0\}$ and f\mbox{}ind $y\in\mathsf{S}_Y$ with $y^*(y)=\|y^*\|$. By hypothesis, there exists $z^*\in \mathsf{S}_Z$ with $z^*(y)=1$. Since $Y$ is smooth, we deduce that $\frac{y^*}{\|y^*\|}=z^*|_Y=i^*(z^*)$. Thus $y^*=i^*(\|y^*\|z^*)\in i^*(Z)$. This shows that $\mathsf{NA}(Y)\subseteq i^*(Z)$. Finally, $(i^*)^{-1}(\mathsf{NA}(Y))\subseteq (i^*)^{-1}(i^*(Z))$. Since $i^*:X^*\to Y^*$ is continuous and onto between Banach spaces, it must be open by the Open Mapping Theorem, therefore $(i^*)^{-1}(\mathsf{NA}(Y))$ is dense in $X^*$ because $\mathsf{NA}(Y)$ is dense in $Y^*$ by the Bishop-Phelps Theorem.

\end{enumerate}

Assume now that there exists a closed subspace $V\neq \{0\}$ of $X$ such that $Y\cap V =\{0\}$ and the natural projection $Y\oplus V \to Y$ has norm $1$. Let $P:Y+ V \to Y$ denote the natural projection and $i:Y+V\hookrightarrow X$ the natural injection. We will take $Z:=(i^*)^{-1}\left(P^*(Y^*)\right)$. We have to show that $Z$ is a nontrivial norm-attaining subspace for $Y$. We will follow several steps:
\begin{itemize}

\item $Z$ is $w^*$-closed in $X^*$. Indeed, $P^*:Y^*\to (Y+V)^*$ is $w^*$-$w^*$ continuous and is a projection, so $P^*(Y^*)$ is $w^*$-closed in $(Y+V)^*$. Next, $i^*:X^*\to (Y+V)^*$ is also $w^*$-$w^*$ continuous, so $(i^*)^{-1}\left(P^*(Y^*)\right)$ is $w^*$-closed in $X^*$.

\item Every $y\in \mathsf{S}_Y$ has a $z^* \in \mathsf{S}_{Z}$ such that $z^*(y)=1$. Indeed, f\mbox{}ix an arbitrary $y\in \mathsf{S}_Y$. By the Hahn-Banach Theorem there exists $y^*\in\mathsf{S}_{Y^*}$ with $y^*(y)=1$. Notice that $P^*(y^* )(y)=y^*(P(y))=y^*(y)=1$ and $\|P^*(y^*)\|= 1$ since $\|P^*\|=\|P\|=1$. In accordance with the Hahn-Banach Extension Theorem we can f\mbox{}ind $z^* \in \mathsf{S}_{X^*}$ such that $i^*(z^*)=z^*|_{Y+V} = P^*(y^* )$. Now observe that $z^*(y)=P^*(y^*)(y)=1$ and $z^*\in Z$.

\item $V\subseteq Z^\top$. Indeed, if $v\in V$ and $z^*\in Z$, then there exists $y^*\in Y^*$ with $i^*(z^*)=P^*(y^*)$ and hence $z^*(v)=P^*(y^*)(v)=y^*(P(v))=y^*(0)=0$. As a consequence, $V\subseteq Z^\top$.

\item $V=Z^\top\cap (Y+V)$. Indeed, let $y+v\in Z^\top$ and f\mbox{}ix an arbitrary $y^*\in Y^*$. Using again the Hahn-Banach Theorem we can f\mbox{}ind $z^*\in X^*$ such that $i^*(z^*)=z^*|_{Y+V}= P^*(y^*)$, then $z^*\in Z$ and hence  $y^*(y)=y^*(P(y+v))=P^*(y^*)(y+v)=z^*(y+v)=0$. The arbitrariness of $y^*$ implies that $y=0$.

\item $Y^\perp \cap Z=\{0\}$ provided that $Y+V=X$. Indeed, in this situation, the natural inclusion $i:Y+V\hookrightarrow X$ is precisely the identity on $X$. Therefore, if we consider the $1$-projection $P:X\to X$ such that $P(X)=Y$ and $\ker(P)=V$, then $Z=P^*(X^*)$. It is clear that $Y^\perp =(I-P^*)(X^*)=\ker(P^*)$, hence $Y^\perp\cap Z=\{0\}$.

\end{itemize}

\end{proof}

As a consequence of Theorem \ref{char}, we obtain a characterization for a closed subspace to be $1$-complemented.

\begin{corollary}
Let $X$ be a Banach space. If $Y$ is a closed subspace of $X$, then there exists   $Z \subseteq X^*$ norm-attaining with $Y^\perp\cap Z=\{0\}$ if and only if $Y$ is $1$-complemented in $X$.
\end{corollary}

The previous corollary provides a characterization of $1$-projective spaces, hence  a suf\mbox{}f\mbox{}icient condition for self-extensionality.

\begin{scholium}
A Banach space $X$ is $1$-projective if and only if every closed subspace $Y$ of $X$ has a norm-attaining $Z \subseteq X^*$  with $Y^\perp\cap Z=\{0\}$. In particular,  $X$  satisf\mbox{}ies the self-extension property.
\end{scholium}

\section{Classic Banach spaces and self-extension}

We already know that the $n$-dimensional Banach spaces $\mathbb{R}^n_2$ and $\mathbb{R}^n_\infty$ have the self-extension property, since $\mathbb{R}^n_2$ is $1$-projective and $\mathbb{R}^n_\infty$ is $1$-injective. Note that $\mathbb{R}^3_1$ is neither $1$-injective nor $1$-projective, however in the next result we show that it has the self-extension property.

\begin{theorem}\label{R3}
$\mathbb{R}^3_1$ satisf\mbox{}ies the self-extension property.
\end{theorem}

\begin{proof}
Let $X = \mathbb{R}^3_1$ and consider an hyperplane $Y \subseteq X$ and a continuous linear operator $T: Y \to Y$. Without loss of generality, we can assume that $e_1 \notin Y$ and $\|T\| = 1$. There exist $r_2, r_3 \in \mathbb{R}$ such that $r_2e_1 - e_2, r_3e_1 - e_3 \in Y$. $r_2 = r_3 = 0$ constitutes a trivial case. If one assumes $r_2 = 0$ and $r_3 \neq 0$, it is enough to def\mbox{}ine $\beta = T(e_1-\frac{1}{r_3}e_3)$ and $S: X \to X$ given by 

$$Se_1 = \frac{|r_3|}{1+|r_3|} \beta$$
$$Se_2 = Te_2$$
$$Se_3 = \frac{-r_3}{1+|r_3|} \beta$$

and it is elementary to conclude that $\|S\| = 1$ and $S$ extends $T$. Consequently, we can assume $r_2\neq0$, $r_3\neq 0$ and def\mbox{}ine $\alpha = T(e_1 - \frac{1}{r_2}e_2)$, $\beta = T(e_1 - \frac{1}{r_3}e_3)$. In order to obtain a more symmetric structure, let $\gamma = 0$, $\mu_2 = \frac{1}{|r_2|}$, $\mu_3 = \frac{1}{|r_3|}$, $\mu_4 = 1$. There exist $s_2, s_3, s_4 \geq 0$ satisfying

$$\|\alpha - \beta\| = \mu_2 + \mu_3 - s_4$$
$$\|\alpha - \gamma\| = \mu_2 + \mu_4 - s_3$$
$$\|\beta - \gamma\| = \mu_3 + \mu_4 - s_2$$

By symmetry, we can assume $s_2 \geq s_3 \geq s_4$. Def\mbox{}ine $z \in X$  as follows\footnote{This expression is suggested by the $R_3$ property in \cite[Page 18]{Lima1977}, which was in turn based on the $3.2.I.P.$ \mbox{def\mbox{}ined in \cite{Lind1964}}.}

$$z_i = \left\{
\begin{array}{lr}
0 & \textrm{if } (\alpha_i-\beta_i)(\alpha_i-\gamma_i) \leq 0 \\
 \alpha_i - \max\{\beta_i, \gamma_i\} & \textrm{if } \alpha_i > \max\{\beta_i, \gamma_i\}\\
 \alpha_i - \min\{\beta_i, \gamma_i\}& \textrm{if } \alpha_i < \min\{\beta_i, \gamma_i\}\\
\end{array}
\right.$$

If we additionally def\mbox{}ine $u = \alpha - \beta - z$, $v = \alpha - \gamma - z$, a computation shows that

$$\|z\| = \frac{1}{2}(\|\alpha-\beta\| + \|\alpha-\gamma\| - \|\beta-\gamma\|) = \mu_2 + \frac{1}{2}(s_2-s_3-s_4)$$
$$\|u\| = \frac{1}{2}(\|\alpha-\beta\| + \|\beta-\gamma\| - \|\alpha-\gamma\|) = \mu_3 + \frac{1}{2}(s_3-s_2-s_4)$$
$$\|v\| = \frac{1}{2}(\|\alpha-\gamma\| + \|\beta-\gamma\| - \|\alpha-\beta\|) = \mu_4 + \frac{1}{2}(s_4-s_2-s_3)$$

Let $w = \alpha - \min\left\{\frac{\mu_2}{\|z\|},1\right\} z$ (we convene that $\frac{\mu_2}{\|z\|} = +\infty$ when $z=0$). Then

$$\|w - \alpha\| = \min\left\{\frac{\mu_2}{\|z\|},1\right\} \|z\| \leq \mu_2$$
$$\|w - \beta\| = \|u + \max\left\{1-\frac{\mu_2}{\|z\|},0\right\} z\| \leq \|u\| + \max\{\|z\|-\mu_2,0\} = \mu_3+\max\left\{-s_4,\frac{s_3-s_2-s_4}{2}\right\} \leq \mu_3$$
$$\|w - \gamma\| = \|v + \max\left\{1-\frac{\mu_2}{\|z\|},0\right\} z\| \leq \|v\| + \max\{\|z\|-\mu_2,0\} = \mu_4+\max\left\{-s_3,\frac{s_4-s_2-s_3}{2}\right\} \leq \mu_4$$

We have proven that, by def\mbox{}ining $S: X \to X$ as the extension of $T$ satisfying $Se_1 = w$, then

$$\|Se_2\| = |r_2|\|Se_1 - \alpha\| \leq 1$$
$$\|Se_3\| = |r_3|\|Se_1 - \beta\| \leq 1$$
$$\|Se_1\| = \|Se_1 - \gamma\| \leq 1$$

which implies that  $\|S\| = 1$.

\end{proof}

However, the self-extension property does not hold for general $\mathbb{R}^n_1$ as the following example shows. For simplicity, we denote $u_k = e_1 - e_k$.

\begin{example}[{\bf $\mathbb{R}^4_1$ fails the self-extension property}]\label{R4}

Let $X = \mathbb{R}^4_1$ and consider $Y = \{x \in X: \sum_{i=1}^4 x(i)=0\}$ generated by the basis $(u_2,u_3,u_4)$. Let $T: Y \to Y$ be given, in terms of such basis, by the following matrix:
$$M_T = \frac{1}{2}\left(
\begin{array}{rrr}
 1 &  1 &  1 \\
 1 & -1 & -1 \\
-1 &  1 & -1 \\
\end{array}\right)$$
Next, it is easy to show that the extreme points of $B_Y$ are exactly $$E_Y:=\left\{\frac{1}{2}(u_i - u_j): i,j \in\{1,2,3,4\}, i \neq j\right\}$$ and from this
$$\|T\| = \max\{\|Tv\|: v \in E_Y\} = 1$$

Now suppose that $S: X \to X$ extends $T$ and let $z=Se_1$. Taking into account that $\|Se_k\| = \|Se_1 - Su_k\| = \|z - Tu_k\|$ for $k \in \{2,3,4\}$, the following holds:
\begin{itemize}
\item $\|Se_1\| = \|z\| \geq z(1)-z(2)+z(3)+z(4)$
\item $\|Se_2\| = \|z - \frac{1}{2}(1,-1,-1,1)\| \geq \frac{1}{2}-z(1)+\frac{1}{2}+z(2)+\frac{1}{2}+z(3)+\frac{1}{2}-z(4)$
\item $\|Se_3\| = \|z - \frac{1}{2}(1,-1,1,-1)\| \geq \frac{1}{2}-z(1)+\frac{1}{2}+z(2)+\frac{1}{2}-z(3)+\frac{1}{2}+z(4)$
\item $\|Se_4\| = \|z - \frac{1}{2}(-1,-1,1,1)\| \geq \frac{1}{2}+z(1)-\frac{1}{2}-z(2)+\frac{1}{2}-z(3)+\frac{1}{2}-z(4)$
\end{itemize}

The sum of the four inequalities gives us $\|Se_1\| + \|Se_2\| + \|Se_3\|  +\|Se_4\| \geq 5$ implying $\|S\| \geq \frac{5}{4}$ and excluding a norm-one extension of $T$.

\end{example}

The following two examples behave similarly to example \ref{R4}, so their properties are stated without proof. They will be used in the next section.

\begin{example}\label{R5}
Let $X = \mathbb{R}^5_1$ and consider $Y = \{x \in X: \sum_{i=1}^5 x(i)=0\}$ generated by the basis $(u_2,u_3,u_4,u_5)$. Let $T: Y \to Y$ be given, in terms of such basis, by the following matrix:
$$M_T = \frac{1}{2}\left(
\begin{array}{rrrr}
 1 &  1 & -1 & -1 \\
 1 &  0 &  1 &  1 \\
-1 & -1 &  0 & -1 \\
-1 & -1 & -1 &  0 \\
\end{array}\right)$$
Then $\|T\| = 1$ and if $S: X \to X$ extends $T$ then $\|S\| \geq \frac{4}{3}$.
\end{example}

\begin{example}\label{R6}
Let $X = \mathbb{R}^6_1$ and consider $Y = \{x \in X: \sum_{i=1}^6 x(i)=0\}$ generated by the basis $(u_2,u_3,u_4,u_5,u_6)$. Let $T: Y \to Y$ be given, in terms of such basis, by the following matrix:

$$M_T = \frac{1}{2}\left(
\begin{array}{rrrrr}
 1 &  1 &  1 &  1 &  0 \\
 0 &  0 &  1 & -1 &  1 \\
 1 &  0 &  0 &  1 &  1 \\
-1 & -1 &  0 &  0 & -1 \\
-1 &  1 & -1 &  0 &  0 \\
\end{array}\right)$$

Then $\|T\| = 1$ and if $S: X \to X$ extends $T$ then $\|S\| \geq \frac{3}{2}$.
\end{example}

We think the following two results are known, however we have not found a proper reference and therefore they are given with complete proofs:

\begin{proposition}
Assume the vectors $u_1, \dots, u_m \in \mathbb{R}^n$ def\mbox{}ine a (necessarily polyhedral) norm $\|\ \|$ in $\mathbb{R}^n$ by means of the expression
$$\|x\| = \max \{|u_1\cdot x|, \dots, |u_m \cdot x|\}$$
Then $(\mathbb{R}^n, \|\ \|)$ is linearly isometric to a subspace of $(\mathbb{R}^m, \|\ \|_\infty)$.
\end{proposition}

\begin{proof}
\mbox{}
Consider the linear isometry $T: (\mathbb{R}^n, \|\ \|) \to (\mathbb{R}^m, \|\ \|_\infty)$ given by $Tx = (u_1\cdot x, \dots, u_m \cdot x)$.
\end{proof}

\begin{corollary}\label{r4intorinf}
$(\mathbb{R}^n, \|\ \|_1)$ is linearly isometric to a subspace of $(\mathbb{R}^{2^{n-1}}, \|\ \|_\infty)$.
\end{corollary}

\begin{proof}
\mbox{}
Just note that $\sum_{i=1}^n |x_i| = \max\{|x_n+\sum_{i=1}^{n-1} \varepsilon_i x_i|: \varepsilon \in \{-1,1\}^{n-1}\}$.
\end{proof}

From Corollary \ref{r4intorinf} it is clear that the self-extension property is not inherited to subspaces, not even to hyperplanes (if the property were inherited to hyperplanes, inductively this would yield that it is inherited to f\mbox{}inite-codimensional subspaces, but this is impossible since $\mathbb{R}^4_1$ is linearly isometric to a subspace of $\mathbb{R}^8_\infty$).

\section{$k-$self-extension and stability}

We def\mbox{}ine a couple of concepts that allow us to study the stability of the self-extension property.

\begin{definition}[{\bf A ``self-extension coef\mbox{}f\mbox{}icient''}]
Let $X$ be a Banach space, we say $X$ is {\bf $k-$self-extensible} if for every subspace $Y \subseteq X$ and every linear, continuous operator $T: Y \to Y$ there exists a linear, continuous $S: X \to X$ such that $S|_Y=T$ and $\|S\| \leq k \|T\|$. Additionally, we def\mbox{}ine $$se(X) = \inf \{k \geq 1: X \textrm{ is } k-\textrm{self-extensible}\}$$ We will say $se(X) = \infty$ when there is no $k$ such that $X$ is $k-$self-extensible.
\end{definition}

Note that examples \ref{R4}, \ref{R5} and \ref{R6} actually show that $se(\mathbb{R}^4_1) \geq \frac{5}{4}$, $se(\mathbb{R}^5_1) \geq \frac{4}{3}$ and $se(\mathbb{R}^6_1) \geq \frac{3}{2}$ respectively.

Let us recall that the Banach-Mazur distance between two isomorphic Banach spaces $X$ and $Y$ can be def\mbox{}ined as

$$d(X,Y) = \inf \{\|U^{-1}\| : U \in {\rm ISO}(X,Y), \|U\| = 1\}$$ \mbox{ }

where ${\rm ISO}(X,Y)$ stands for the set of isomorphisms from $X$ to $Y$. It is well-known that $d$ is symmetric.

\begin{proposition}[{\bf On stability of $k-$self-extension}]\label{stabilityofself}
Let $X, Y$ be isomorphic Banach spaces. It holds that $se(X) \leq se(Y)d(X,Y)^2$.
\end{proposition}

\begin{proof}
We can assume $se(Y) < \infty$, otherwise it is trivial. Let $U: X \to Y$ be an isomorphism such that $\|U\| = 1$.
Let $Z \subseteq X$ be a subspace and $T: Z \to Z$ a linear, continuous operator. Then $W=U(Z)$ is a subspace of $Y$ and $UTU^{-1}: W \to W$ is linear and continuous. Thus, there exists a linear, continuous extension $S: Y \to Y$ of $UTU^{-1}$ such that $\|S\| \leq se(Y)\|UTU^{-1}\| \leq se(Y)\|U^{-1}\|\|T\|$ and $U^{-1}SU: X \to X$ is an extension of $T$ satisfying $$\|U^{-1}SU\| \leq \|U^{-1}\|se(Y)\|U^{-1}\|\|T\| = se(Y)\|U^{-1}\|^2\|T\|.$$
We conclude by taking inf\mbox{}imum on $U$.
\end{proof}

By using $1$-self-extendability, we can discuss some topological properties of the set of self-extensible spaces isomorphic to a given one. Namely,

\begin{corollary}\label{corminkowski}
Let $X$ be a Banach space and let $\mathcal{C}$ be the set of Banach spaces isomorphic to $X$. Assume $\mathcal{C}$ is pseudometrized by $\log d$ and consider the sets
$$\mathcal{A} = \{Y \in \mathcal{C}: Y \textrm{ has the self-extension property, i.e. is }1-\textrm{self-extensible}\}$$
$$\mathcal{B} = \{Y \in \mathcal{C}: se(Y) = 1 \}$$
We have
\begin{itemize}
\item $\mathcal{B}$ is closed and $\mathcal{A} \subseteq \mathcal{B}$.
\item If $X$ is f\mbox{}inite-dimensional then $\mathcal{B}$ is compact and $\mathcal{A} = \mathcal{B}$.
\end{itemize}
\end{corollary}

\begin{proof}
The closedness of $\mathcal{B}$ is due to Proposition \ref{stabilityofself}, whereas $\mathcal{A} \subseteq \mathcal{B}$ is deduced from the def\mbox{}initions.

Assume now that $X$ is f\mbox{}inite-dimensional. It is well known that $\mathcal{C}$ is a compact metric space and therefore $\mathcal{B}$ is compact. Finally, let $Y \in \mathcal{B}$ and $Z \subseteq Y$ be a subspace. If $T: Z \to Z$ is a continuous linear mapping then for each $n$ there exists an extension of $T$, say $S_n: Y \to Y$, such that $\|S_n\| \leq (1+\frac{1}{n})\|T\|$. Since $\mathcal{L}(Y,Y)$ is f\mbox{}inite-dimensional, $S_n$ has a subsequence converging to some $S: Y \to Y$ which is necessarily an extension of $T$ satisfying $\|S\| = \|T\|$.
\end{proof}

Now we will need the fact that $d(\mathbb{R}^4_1,\mathbb{R}^4_p) = 4^{1-1/p}$ when $p \leq 2$ (this can be found in the book \cite[Page 280]{Tom1989}, where the results are attributed to Gurarij, Kadets and Macaev in \cite{GKM1965} and \cite{GKM1966}).

\begin{corollary}[{\bf particularization for $\mathbb{R}^4_p$}]\label{particr4}
Let $X$ be a 4-dimensional normed space such that $d(X,\mathbb{R}^4_1) < \frac{\sqrt{5}}{2}$. Then $X$ does not satisfy the self-extension property. In particular, $\mathbb{R}^4_p$ is not self-extensible for every $p$ such that $1 \leq p < \frac{1}{1-\log_4\frac{\sqrt{5}}{2}} = 1.0875 \dots$.
\end{corollary}

\begin{proof}\mbox{ }

\begin{itemize}
\item Suppose that $X$ is self-extensible, that is, $se(X) = 1$; then, the previous proposition implies $se(\mathbb{R}^4_1) < \frac{5}{4}$ which is false.
\item Elementary computation leads us to $1 \leq p < \frac{1}{1-\log_4\frac{\sqrt{5}}{2}}$ if and only if $1 \leq 4^{1-1/p} < \frac{\sqrt{5}}{2}$.
\end{itemize}
\end{proof}

Corollary \ref{particr4} could be restated for ${\mathbb{R}}^5_1$ or ${\mathbb{R}}^6_1$ by using the aforementioned examples, but we will make no use of that. Some consequences of the previous results are:

\begin{itemize}
\item The self-extension property is not transferred to preduals or duals (since $\mathbb{R}^4_1$ lacks the property whereas $\mathbb{R}^4_\infty$ has it).
\item The self-extension property is not deduced from uniform convexity plus uniform smoothness (since $\mathbb{R}^4_{1.08}$ has both geometrical properties and lacks self-extension).
\item The self-extension property is not inherited to subspaces, not even to hyperplanes (already mentioned after Corollary \ref{r4intorinf}).
\end{itemize}

Finally, let us mention some open questions which appear naturally:

\begin{itemize}
\item Suppose the Banach space $X$ has the following property: Whenever $Y$ is a hyperplane of $X$ and $T: Y \to Y$ is a continuous linear operator, there exists a continuous linear extension $S: X \to X$ satisfying $\|T\| = \|S\|$. Can we deduce that $X$ has the self-extension property?
\item Is there a Banach space $X$ satisfying $se(X) = \infty$?
\item In Corollary \ref{corminkowski}, is it true that $\mathcal{A} = \mathcal{B}$ without any additional hypothesis on $X$? In other words, does $X$ have the self-extension property if $se(X)=1$?
\end{itemize}

\section{Examples of inf\mbox{}inite dimensional Banach spaces failing the self-extension property}

Although the self-extension property is not inherited to hyperplanes, it is nonetheless inherited to $1$-complemented subspaces:

\begin{proposition}\label{her}
Let $X$ be a Banach space and $Y$ a $q$-complemented subspace of $X$. Then $se(Y) \leq q se(X)$. Moreover, if $X$ has the self-extension property and $q=1$ then $Y$ also enjoys it.
\end{proposition}

\begin{proof}
Assume $X$ is $k$-self-extensible for some $k \geq 1$. Let $Z$ be a closed subspace of $Y$ and $T\in\mathcal{L}(Z)$. There exists $R\in\mathcal{L}(X)$ such that $\|R\| \leq k\|T\|$ and $R|_Z=T$. If $P:X\to X$ is a $q$-projection onto $Y$, the mapping $S = P \circ R|_Y: Y \to Y$ satisf\mbox{}ies $\|S\| \leq \|P\|\|R\| \leq q\|R\| \leq qk\|T\|$ and $Sz = PRz = PTz = Tz$ for every $z \in Z$.
\end{proof}

Since every $\mathbb{R}^n_1$ is $1$-complemented in $\mathbb{R}^{n+1}_1$ and these spaces are all $1$-complemented in $\ell_1$, joining proposition \ref{her} with example \ref{R6} we immediately obtain the following result.

\begin{corollary}
The real $\ell_1$ fails the self-extension property, and moreover
$$\frac{3}{2} \leq se(\mathbb{R}^6_1) \leq se(\mathbb{R}^7_1) \leq \dots \leq se(\ell_1)$$
\end{corollary}

We think that the previous inequalities are indeed equalities. In other words,

\begin{conjecture}
$se(\ell_1) = \frac{3}{2}$.
\end{conjecture}

The $c_0$ case will be studied in the next section. Now we will continue scratching more consequences out of Proposition \ref{her} in terms of failing self-extendability, isomorphically speaking.

\begin{theorem}
Every real Banach space of dimension greater than or equal to three can be equivalently renormed to fail the self-extension property.
\end{theorem}

\begin{proof}
Let $X$ be a real Banach space with $\dim(X)\geq 3$. Consider any $3$-dimensional subspace $Y$ of $X$. We can trivially renorm $X$ equivalently in such a way that $Y$ is $1$-complemented in $X$ and isometric to the non-self-extensible $3$-dimensional example given in \cite{LPR1998}. Finally observe that $X$ endowed with this new equivalent norm cannot enjoy the self-extension property because $Y$ is a $1$-complemented subspace of $X$ that fails it (see Proposition \ref{her}).
\end{proof}

\section{Finite self-extension property}

We already know that $\ell_\infty$ enjoys the self-extension property because it is $1$-injective (it is a $\mathcal{C}(K)$ with $K$ extremally disconnected). However, extensions of operators def\mbox{}ined on closed subspaces of $\ell_\infty$ can be easily constructed explicitly in virtue of the Hahn-Banach Extension Theorem. Indeed,  if $Y$ is a closed subspace of $\ell_\infty$ and $S:Y\to Y$ is linear and continuous, then  it suf\mbox{}f\mbox{}ices to def\mbox{}ine $$\begin{array}{rrcl} T : & \ell_\infty &\to& \ell_\infty \\ & x & \mapsto & T(x):=(\gamma_n(x))_{n\in\mathbb{N}},\end{array}$$ where, for every $n\in \mathbb{N}$, $\gamma_n\in\left(\ell_\infty\right)^*$ is a extension of $\delta_n|_Y\circ S\in Y^*$ with the same norm, that is, $\gamma_n|_Y=\delta_n|_Y\circ S$ and $\|\gamma_n\|=\left\|\delta_n|_Y\circ S\right\|$. Here, $\delta_n$ stands for the coordinate evaluation functional on $\ell_\infty$ $$\begin{array}{rrcl} \delta_n : &\ell_\infty &\to&\mathbb{K} \\ & x & \mapsto & \delta_n(x):=x(n).\end{array}$$

\begin{definition}
A Banach space $X$ is said to enjoy the f\mbox{}inite self-extension property if for every f\mbox{}inite dimensional subspace $Y$ of $X$ and every linear and continuous $T:Y\to Y$, there exists a continuous linear extension of $T$ preserving its norm. In other words, there exists $S:X\to X$ linear and continuous such that $S|_Y=T$ and $\|S\|=\|T\|$.
\end{definition}

Following the same ideas described at the beginning of this section, we will prove that $c_0$ has the f\mbox{}inite self-extension property.

\begin{definition}
A subspace $Y$ of $\ell_\infty$ is called $\delta$-null provided that $(\delta_n|_Y)_{n\in\mathbb{N}}$ converges to $0$ in $Y^*$.
\end{definition}

Trivial examples of $\delta$-null subspaces of $\ell_\infty$ are $\ell_\infty^n$ for all $n\in\mathbb{N}$. In fact, a standard compactness argument shows that the $\delta$-null subspaces of $c_0$ are exactly its f\mbox{}inite dimensional subspaces.

\begin{theorem}
If $Y$ is a $\delta$-null subspace of $c_0$, then every $T\in \mathcal{L}(Y)$ can be extended to a $S \in \mathcal{L}(c_0)$ preserving its norm. In particular, $c_0$ satisf\mbox{}ies the f\mbox{}inite self-extension property.
\end{theorem}

\begin{proof}
By the Hahn-Banach Theorem, we can extend every $\delta_n|_Y\circ T$ to a $\gamma_n\in c_0^*$ such that $\|\gamma_n\|=\|\delta_n|_Y\circ T\|\leq \|\delta_n|_Y\|\|T\|\leq \|T\|$. Consider $$\begin{array}{rrcl} S : & c_0 &\to& c_0 \\ & x & \mapsto & S(x):=(\gamma_n(x))_{n\in\mathbb{N}}.\end{array}$$ Observe the following:
\begin{itemize}
\item $S$ is well-def\mbox{}ined. Indeed, $|\gamma_n(x)|\leq \|\gamma_n\|\|x\|\leq \|\delta_n|_Y\|\|T\|\|x\|\to 0$.
\item $S|_Y=T$. Indeed, it holds because $\gamma_n|_Y=\delta_n|_Y\circ T$.
\item $\|S\|=\|T\|$. Indeed, $\|T\|= \|S|_Y\|\leq\|S\|\leq \sup_{n\in\mathbb{N}}\|\gamma_n\|\leq \|T\|$.
\end{itemize}
\end{proof}

\section*{Funding}

F.J. Garc\'{i}a-Pacheco and F. Rambla-Barreno were supported by Junta de Andaluc\'{\i}a FQM-257, Plan Propio de la Universidad de C\'adiz and FEDER/Ministerio de Ciencia, Innovaci\'on y Universidades - Agencia Estatal de Investigaci\'on PGC2018-101514-B-I00. F.J. Garc\'{i}a-Pacheco was supported by the 2014-2020 ERDF Operational Programme and by the Department of Economy, Knowledge, Business and University of the Regional Government of Andalusia, grant number FEDER-UCA18-105867. F. Rambla-Barreno was supported by the Department of Economy, Knowledge, Business and University of the Regional Government of Andalusia, grant number FEDER-UCA18‐108415.

\section*{Conflicts of interest/Competing interests}

The authors declare that they have no conflict of interests or competing interests.

\section*{Author’s contributions}

All authors contributed equally and signif\mbox{}icantly in writing this article. All authors read and approved the f\mbox{}inal manuscript.

\end{document}